\documentclass[review, 11pt]{elsarticle}

\usepackage{graphicx, amsmath, amssymb, amsthm, color}
\usepackage[margin=2.5cm]{geometry}
\usepackage{mathtools}
\usepackage{float}
\usepackage{caption, keyval, subfig}
\usepackage{commath}              
\graphicspath{{./figs/}}

\usepackage{commath}              
\usepackage{algorithmic}			
\usepackage{algorithm}
\usepackage{placeins}                
\usepackage{hyperref}                 
\usepackage{enumerate}

\usepackage{setspace}
\theoremstyle{plain}
\newtheorem{theorem}{Theorem}[section]                
\newtheorem{proposition}[theorem]{Proposition}       
\newtheorem{corollary}[theorem]{Corollary}

\newtheorem{definition}[theorem]{Definition}
\newtheorem{example}[theorem]{Example}
\newtheorem{remark}[theorem]{Remark}

\newcommand{\rba}[3]{\mathfrak{M}_{#1} \left( {#2},{#3} \right) }         
\newcommand{\srba}[3]{\mathfrak{M}^{S}_{#1} \left( {#2},{#3} \right) }  

\begin{document}

\begin{frontmatter}

\title{Manifold-valued subdivision schemes \\ based on geodesic inductive averaging}  


\author[mymainaddress]{Nira Dyn}
\ead{niradyn@post.tau.ac.il}

\author[mysecondaryaddress]{Nir Sharon\corref{mycorrespondingauthor}}
\cortext[mycorrespondingauthor]{Corresponding author}
\ead{nsharon@math.princeton.edu}

\address[mymainaddress]{School of Mathematical Sciences, Tel Aviv University, Tel Aviv, Israel}
\address[mysecondaryaddress]{The Program in Applied and Computational Mathematics, Princeton University, Princeton NJ, USA}

\begin{abstract} 
Subdivision schemes have become an important tool for approximation of manifold-valued functions. In this paper, we describe a construction of manifold-valued subdivision schemes for geodesically complete manifolds. Our construction is based upon the adaptation of linear subdivision schemes using the notion of repeated binary averaging, where as a repeated binary average we propose to use the geodesic inductive mean. We derive conditions on the adapted schemes which guarantee convergence from any initial manifold-valued sequence. The definition and analysis of convergence are intrinsic to the manifold. The adaptation technique and the convergence analysis are demonstrated by several important examples.
\end{abstract}

\begin{keyword}
Manifold-valued subdivision scheme \sep convergence \sep contractivity \sep displacement-safe scheme \sep inductive geodesic mean
\MSC[2010] 40A30 \sep 65D99 \sep 41A05
\end{keyword}

\end{frontmatter}







\section{Introduction}

In recent years methods which model certain modern data as manifold data have been developed. An example of such data is the set of orientations of an aircraft, as recorded by its black box. This time series can be interpreted as data sampled from a function mapping a real interval (the time) to the Lie group of orthogonal matrices (the orientations), see e.g., \cite{donoho}. Yet, classical methods for approximation cannot cope with manifold-valued functions. For instance, there is no guarantee that linear approximation methods such as polynomial or spline interpolation produce manifold values, due to the non-linearity of manifolds.

Contrary to the development of classical approximation methods and numerical analysis methods for real-valued functions, the development in the case of manifold-valued functions, which is rather recent, was mainly concerned in its first stages with advanced numerical and approximation processes, such as geometric integration of ODE on manifolds, e.g. \cite{iserles2000lie}, subdivision schemes on manifolds, e.g. \cite{DynSharonGlobal2015,WallnerNiraProx,ThomasLie1}, and wavelets-type approximation on manifolds, e.g. \cite{donoho,wallnerconvergent}. In this paper we focus on subdivision schemes.

Subdivision schemes were created originally to design geometrical models \cite{ChaikinPaper}. Soon, they were recognized as methods for approximation \cite{dubucDeslauriers1,4Point}. The important advantage of these schemes is their simplicity and locality. Namely, they are defined by repeatedly applying simple and local arithmetic averaging. This feature enables the extension of subdivision schemes to more abstract settings, such as matrices \cite{NirUri}, sets \cite{dyn2001spline}, curves \cite{itai2011generating}, and nets of functions \cite{conti2011analysis}.

For manifold valued data, \cite{WallnerNiraProx} introduced the concept of adapting linear subdivision schemes to manifold values, in particular for Lie groups data. This paper initiated a new direction of research on manifold-valued subdivision schemes, see e.g., \cite{UriNir,wallnerconvergent,ThomasLie1}. The adaptation of linear subdivision schemes in this paper is done by rewriting the refinement rules in repeated binary average form, and then replacing each binary average with a weighted binary geodesic average, see e.g., \cite{NirUri,WallnerNiraProx}. 

A weighted geodesic average is a generalization of the arithmetic average $(1-t)a+tb$ in Euclidean spaces, and is defined for any weight $t \in [0,1]$ as the point on the geodesic curve between the two points to be averaged, which divides it in the ratio $\frac{t}{1-t}$ (for $t=\frac{1}{2}$ it is the midpoint). Furthermore, on several manifolds, it can also be extended to weights outside $[0,1]$, by extrapolating the geodesic curve of two points beyond the points, see e.g., \cite{UriNir}. This facilitates the adaptation of interpolatory subdivision schemes which typically involve averages with negative weights. The geodesic average is also well-defined in more general spaces known as geodesic metric spaces, see e.g., \cite{bridson1999metric}, and our adaptation process and most of its analysis are also valid there.

The adaptation method proposed in this paper is for values from geodesically complete manifolds. It uses a specific form of repeated binary averaging -- the geodesic inductive mean, which enables to deduce the contractivity of the adapted schemes obtained from the well-known interpolatory $4$-point scheme \cite{4Point}, the $6$-point Dubuc-Deslauriers scheme \cite{dubucDeslauriers1}, and the first four B-spline subdivision schemes (see e.g., \cite{NiraScheme}). The contractivity is important since it is closely related to the fundamental question of convergence. 

Many results in the literature of the past few years concerning the convergence and smoothness of adapted subdivision schemes, are based on proximity conditions (see \cite{WallnerNiraProx}). A proximity condition describes a relation between the operation of an adapted subdivision scheme to the operation of its linear counterpart. Since local manifold data are nearly in a Euclidean space, the convergence results based on proximity conditions actually show that the generated values of an adapted scheme are not ``too far" from those generated by its original linear scheme. Thus, these results are valid only for ``dense enough data", which is, in general, a condition that is hard to quantify and depends on the properties of the underline manifold (such as its curvature).

Recently, a progress in the convergence analysis is established by several papers which address the question of convergence from all initial data. Such a result is presented in \cite{ebner2013convergence} for the adaptation of schemes with non-negative mask coefficients to Hadamard spaces. Results for geodesic based subdivision schemes, as well as other adaptation methods, are derived in \cite{NirUri} for the manifold of positive definite matrices. For the case of interpolatory subdivision schemes there are such convergence results for several different metric spaces \cite{UriNir, kels2013subdivision, wallnerconvergent}. In this paper, we present a condition, termed displacement-safe, guaranteeing that contractivity leads to convergence, for all initial data. The displacement-safe condition requires the values after one refinement to be not too far away from the values before the refinement. First we show that our adapted schemes are displacement-safe. Then, we demonstrate the analysis of contractivity on several adapted subdivision schemes, obtained from popular linear schemes, with masks of relatively small support. The contractivity guarantees the convergence of these schemes from all initial data.

The paper is organized as follows. We start in Section~\ref{sec:theoreticalBack} by providing a short survey of the required background, including a summary on linear subdivision schemes, a brief review on manifolds and geodesics, and several popular approaches to the adaptation of those schemes to manifold-valued data. In Section~\ref{sec:contractivity_2_convergence} we introduce the displacement-safe condition which links contractivity and convergence. Section~\ref{sec:our_adaptation} presents our method of adaptation and the proof showing that the adapted schemes are displacement-safe. We conclude the paper in Section~\ref{sec:examples} with the adaptation of  few popular schemes, and prove their convergence from all initial manifold data. 


\section{Theoretical background and notation} \label{sec:theoreticalBack}

We start by providing a few background facts together with notation on subdivision schemes, on manifolds, and on the adaptation of subdivision schemes to manifold data.


\subsection{Linear univariate subdivision schemes} \label{subsec:classical_subdivision_schemes}

In the functional setting, a univariate subdivision schemes, $\mathcal{S}$, operates on a real-valued sequence $\mathbf{f} = \{ f_i \in \mathbb{R} \mid i \in \mathbb{Z} \}$, by applying refinement rules that map $\mathbf{f}$ to a new sequence $\mathcal{S}(\mathbf{f})$ associated with the values in $\frac{1}{2} \mathbb{Z}$. This process is repeated infinitely and results in values defined on the dense set  of dyadic real numbers. In case the values generated from any $\mathbf{f}$ by this process converge uniformly at the dyadic points to values of a continuous function, we term the subdivision scheme convergent, see e.g., \cite{Dyn:2002:AOC}. A necessary and sufficient condition for the convergence of a subdivision scheme is that the sequence $\operatorname{PL}_k$, $k \in \mathbb{N}$, consists of piecewise linear interpolants to each $k$-th refined data $\{ (i2^{-k},(\mathcal{S}^kf)_i \mid i \in \mathbb{Z} \}$, is a Cauchy sequence in the uniform norm. We denote the limit of a convergent subdivision scheme, with the refinement rules $\mathcal{S}$, generated from the initial data $\mathbf{f}$ by $\mathcal{S}^{\infty}(\mathbf{f})$.

A linear univariate subdivision $\mathcal{S}$ is defined by the refinement rules,
\begin{equation} \label{eqn:classical_subdivision_scheme}
\mathcal{S}(\mathbf{f})_j = \sum_{i \in \mathbb{Z}} a_{j-2i} f_i , \quad j \in \mathbb{Z},
\end{equation}
with a finitely supported mask $\mathbf{a} = \{ a_i \} $. The refinement rules \eqref{eqn:classical_subdivision_scheme} can be written as the two rules
\begin{equation}\label{eqn:TwoPartsLinearSchemes}
  \mathcal{S}(\mathbf{f})_{2j} = \sum_{i = -\ell}^{u} a_{2i} f_{j-i}
  \qquad\quad\text{and}\quad\qquad
    \mathcal{S}(\mathbf{f})_{2j+1} = \sum_{i = -\ell}^{u} a_{2i+1} f_{j-i} ,
\end{equation}
where the coefficients of the mask are $\left(\ldots,0,a_{-2\ell},a_{-2\ell+1},\ldots,a_{2u},a_{2u+1},0,\ldots \right)$. A subdivision scheme with fixed refinement rules is termed uniform and stationary.  A subdivision scheme is termed interpolatory if $\mathcal{S}(\mathbf{f})_{2j} = f_j$, for all $j \in \mathbb{Z}$. The compact support of of the mask $\mathbf{a}$ ensures that any value $\mathcal{S}(\mathbf{f})_j$ depends only on a finite numbers of elements of $\mathbf{f}$ adjacent to $f_{\lfloor \frac{j}{2} \rfloor}$. This property is also inherited by the limit of the subdivision process. Therefore, subdivision schemes are local operators.

A necessary condition for the convergence of a subdivision scheme with the refinement rules \eqref{eqn:TwoPartsLinearSchemes} (see e.g. \cite{NiraScheme}), is
\begin{equation} \label{eqn:necessary_con_4_convergence}
\sum_{i \in \mathbb{Z}} a_{2i}=\sum_{i \in \mathbb{Z}} a_{2i+1} = 1 .
\end{equation}

In this paper we discuss the adaptation of linear univariate subdivision schemes from real numbers to manifold-valued data. We confine the adaptation to linear schemes with masks satisfying \eqref{eqn:necessary_con_4_convergence}. To distinguish between subdivision schemes operating on numbers (or vectors) to those operating on manifold values, we denote by $\mathbf{f} = \{ f_i \}_{i \in \mathbb{Z}}$ and $\mathbf{p} = \{ p_i \}_{i \in \mathbb{Z}}$ the data in Euclidean spaces and in real manifolds, respectively.


\subsection{On manifolds and geodesics} \label{subsec:manifolds_and_geodesics}

The Riemannian metric for a connected manifold $\mathcal{M}$ is a collection of symmetrical positive-definite bilinear forms on the tangent spaces which vary smoothly on $\mathcal{M}$. The length of a curve $\gamma$ on $\mathcal{M}$ is given by integrating along the curve the norm induced by the Riemannian metric. An important conclusion is that any connected Riemannian manifold is a metric space. Specifically, the intrinsic distance between two points $p_0,p_1 \in \mathcal{M}$, also called the Riemannian distance and denoted by $d(p_0,p_1)$, is defined as the infimum of the lengths of all curves connecting $p_0$ and $p_1$. Geodesics (or geodesic curves) are derived from the basic question of finding the above shortest curve, joining two arbitrary points. For two points $f_0$ and $f_1$ in a Euclidean space the shortest curve is simply the segment 
\begin{equation} \label{eqn:arithmetic_mean}
 (1-t)f_0+tf_1 , \quad t \in [0,1] . 
\end{equation}
A geodesic curve is defined as the solution to the geodesic Euler-Lagrange equations. It turns out that any shortest path between two points must be a geodesic, and it is termed a minimal geodesic. As a solution to these differential equations, the geodesic curve at a point $p_0 \in \mathcal{M}$ with a given initial direction from the tangent space at $p_0$ is unique. In fact, there exists a radius called the injectivity radius at $p_0$, $\rho(p_0)$ such that the geodesics are unique and minimal in the ``injectivity disc" of $p_0$, that is $\{p \in \mathcal{M} \colon d(p,p_0)<\rho(p_0) \}$.

In connected Riemannian manifolds, the Hopf-Rinow theorem characterizes the conditions which guarantee that geodesic curves connecting any two points are globally well defined. These manifolds are the complete Riemannian manifolds or geodesically complete manifolds. In a geodesically complete manifold there is a positive lower bound of all the injectivity radii of its points, where geodesics are minimal and unique. Nevertheless, despite of their global definition not every geodesic can be extended as a minimal geodesic beyond the injectivity disc. Another result from the Hopf-Rinow theorem is that a geodesically complete manifold is complete as a metric space $(\mathcal{M},d)$, which is essential for our convergence analysis. It is worth mentioning that any compact restriction of a general Riemannian manifolds is complete, and therefore the results in this paper, in view of the locality of subdivision schemes, are relevant to a wide class of manifolds. For more details on geodesic complete manifolds see e.g., \cite[Chapter 13.3]{gallier2012notes} and references therein. 

Geodesics have a major role in our adaptation process. Therefore, our prototype manifolds are complete Riemannian manifolds. Henceforth, we denote by $\mathcal{M}$ a complete Riemannian manifold and by $d(\cdot,\cdot)$ its associated Riemannian distance. Let $\mathcal{G}(t;p_0,p_1)$, $t \in [0,1]$ denote a minimal geodesic curve connecting two points $p_0,p_1$ in $\mathcal{M}$, such that $\mathcal{G}(i;p_0,p_1)=p_i$, $i=0,1$ and
\begin{equation*} 
  d\left(p_0,\mathcal{G}(t; p_0,p_1)\right) =  td\left(p_0,p_1\right) , 
\end{equation*}
or equivalently
\[ d\left(\mathcal{G}(t; p_0,p_1),p_1\right) =  (1-t)d\left(p_0,p_1\right) , \quad t \in [0,1] . \]

We define the geodesic average of $p_0$ and $p_1$ with weight $t$ as
\[ M_t(p_0,p_1) = \mathcal{G}(t; p_0,p_1) . \]
in analogy to the arithmetic average \eqref{eqn:arithmetic_mean}. Thus, $M_t$ has the metric property
\begin{equation} \label{eqn:metric_property}
d\left(p_0,M_t(p_0,p_1)\right) =  td\left(p_0,p_1\right) , \quad d\left(M_t(p_0,p_1),p_1\right) =  (1-t)d\left(p_0,p_1\right) , \quad t \in [0,1] . 
\end{equation}
In case the minimal geodesic is not unique, we choose one in a canonical way (see e.g., \cite{dyn2001spline}). In the adaptation of schemes with negative mask coefficients, such as interpolatory schemes, we also use $M_t$ with values of $t$ outside $[0,1]$, but close to it. In these cases the metric property \eqref{eqn:metric_property} is modified by replacing $t$ and $1-t$ by their absolute values. 

There are spaces, more general than Riemannian manifolds, where any two points in the space can be connected by a curve satisfying the metric property. Such are the geodesic metric spaces, see e.g., \cite{bridson1999metric}. In these spaces, the differential structure is missing and the geodesic curve is defined by the metric property. Clearly, this definition agrees with the geodesic curve on Riemannian manifolds.

 
\subsection{Adaptation methods} \label{subsec:Subdivision_schemes_based_RBA}


There are several different methods for the adaptation of the refinement rules in \eqref{eqn:TwoPartsLinearSchemes} to manifold data. Here we present shortly three ``popular" methods, all ``intrinsic" to the manifold and independent of the ambient Euclidean space.

The first method is based on the log-exp mappings, and consists of three steps. In case of a Lie-group these steps are (see e.g., \cite{ThomasLie1}): (i) projecting the points in $\mathcal{M}$ taking part in the refinement rule into the corresponding Lie algebra, (ii) applying the linear refinement rule on the projected samples in the Lie algebra, (iii) projecting the result back to the Lie group. There are several computational difficulties in the realization of this ``straightforward" idea, mainly in the computation of the logarithm and exponential maps, see e.g., \cite{shingel2009interpolation}. 

The same idea applies for general manifolds but with the Lie algebra replaced by the tangent space at a ``base" point on the manifold, chosen in the neighbourhood of the points taking part in the refinement rule. The inherent difficulty in this approach is the choice of the base point, see e.g., \cite{wallnerconvergent}.

The second method is based on repeated binary geodesic averages. The refinement rules $\mathcal{S}$ of the form \eqref{eqn:TwoPartsLinearSchemes} satisfying \eqref{eqn:necessary_con_4_convergence}, can be written in terms of repeated weighted binary averages in several ways \cite{WallnerNiraProx}. Using one of these representations of $\mathcal{S}$, and replacing each binary average between numbers, by a geodesic average between two points on the manifold, one gets an adaptation of $\mathcal{S}$ to the manifold. For an example see \cite{NirUri}. The difficulty in this approach is the choice of the form of the repeated binary averages. In this paperwe suggest such a form and discuss its advantage.
 
The third method for the adaptation of $\mathcal{S}$ is based on the Riemannian center of mass. Interpreting each sum in \eqref{eqn:TwoPartsLinearSchemes} as a weighted affine average, one replaces each average by the corresponding Riemannian center of mass. The inherent difficulty in this approach is that the Riemannian center of mass is not known explicitly and has to be computed by iterations, see e.g., \cite{grohs2012quasi}. This center of mass is defined in \eqref{eqn:Karcher_Mean} and is briefly discussed in Subsection~\ref{subsec:ouradaptation}.


\section{From contractivity to convergence}  \label{sec:contractivity_2_convergence}

The analysis of adapted subdivision schemes in many papers is based on the method of proximity, introduced in \cite{WallnerNiraProx}. This analysis uses conditions that indicate the proximity of the adapted refinement rule $\mathcal{T}$ to its corresponding linear refinement rule $\mathcal{S}$. The basic proximity condition is 
\begin{equation} \label{eqn:proximity_condition}
d \left( \mathcal{S}(\mathbf{p}),\mathcal{T}(\mathbf{p}) \right)  \le \nu  \left( \delta(\mathbf{p}) \right)^2 ,
\end{equation}
with $\nu>0$ and where
\begin{equation} \label{eqn:delta_def}
\delta(\mathbf{p}) = \sup_{i \in \mathbb{Z}} d(p_i,p_{i+1}) < \infty .
\end{equation}
If $\mathcal{S}$ is a refinement rule of a linear convergent scheme that generates $C^1$ limits, then condition \eqref{eqn:proximity_condition} together with another technical assumption on the refinement rule $\mathcal{S}$, leads to the conclusion that $ \mathcal{T}$ also generates $C^1$ limits, if it converges. The weakness of the proximity method is that convergence is only guaranteed for ``close enough" initial data points. This requirement is typically not easy to quantify as it depends on the manifold and its curvature. Thus, there is much greater benefit in using the proximity method for smoothness analysis when convergence is already assured. For example, the $C^1$ smoothness of adapted schemes based on geodesic averages which satisfy the proximity condition \eqref{eqn:proximity_condition}, is established in \cite{WallnerNiraProx}. Our convergence results  directly indicates the $C^1$ smoothness of the limits of the adapted subdivision schemes by repeated geodesic averages, since such schemes in manifolds with a globally bounded curvature, satisfy \eqref{eqn:proximity_condition} \cite{WallnerNiraProx}. Henceforth, we do not address the question of smoothness and concentrate on convergence, starting with general results on convergence.

First, we provide a formal definition of the contractivity property in the manifold setting.
\begin{definition} 
A manifold-valued subdivision scheme $\mathcal{T}$ has a contractivity factor $\mu$, if there exists $\mu<1$ such that
\[ \delta\left( \mathcal{T}(\mathbf{p}) \right)  \le \mu \delta(\mathbf{p})   ,  \]
for any data $\mathbf{p}$ on the manifold.
\end{definition}
For linear subdivision schemes contractivity of the refinement rules implies the convergence of the subdivision schemes from any initial data, see e.g. \cite{NiraScheme}. For general schemes, a contractivity factor is not sufficient for convergence. This can be easily seen by adding a small constant to each refinement rule of a converging linear subdivision scheme. Therefore, we introduce an additional condition which together with contractivity guarantees convergence. This condition is similar to a condition in \cite{marinov2005geometrically}, and is termed ``displacement-safe" after the latter.

\begin{definition}[Displacement-safe] 
We say that a subdivision scheme $\mathcal{T}$ is displacement-safe if
\begin{equation} \label{eqn:contractivity_extra_condition}
d(\mathcal{T}(\mathbf{p})_{2j},p_j )  \le C \delta(\mathbf{p} ) ,\quad j \in \mathbb{Z} .
\end{equation}
for any sequence of manifold data $\mathbf{p}$, where $C$ is a constant independent of $\mathbf{p}$.
\end{definition}

\begin{remark} \label{rem:DS_condition}
Two additional comments about the displacement-safe condition:
\begin{enumerate}[(a)]
\item
All converging linear subdivision schemes (for numbers) are displacement-safe. This follows from the necessary condition for convergence \eqref{eqn:necessary_con_4_convergence} and the linearity of the schemes, 
\[     
\abs{ \mathcal{S}(\mathbf{f})_{2i}-f_{i}} = \abs{\sum_{j} a_{2i-2j}f_j - f_i} =  \abs{\sum_{j} a_{2i-2j}(f_j - f_i)} \le  \left( \sum_{j} \abs{a_{2j}} \right) C \delta(\mathbf{f}) ,
\]
where $C$ depends on the size of the support of the mask $\{a_j\}$. 
\item
Relation \eqref{eqn:contractivity_extra_condition} clearly holds for manifold-valued interpolatory schemes that satisfy 
\begin{equation} \label{eqn:interpolatory_adapted_ref}
\mathcal{T}(\mathbf{p})_{2j}=\mathbf{p}_j , \quad j \in \mathbb{Z}. 
\end{equation}
\end{enumerate}	
\end{remark}

For the convergence analysis we follow the classical tools and extend the piecewise linear polygon to manifold-valued data.
\begin{definition}[Piecewise geodesic curve] \label{def:Piecewise_geodesic}
Let $\mathbf{p} \subset \mathcal{M}$ be a sequence of manifold data. For any non-negative integer $k$, we define the piecewise geodesic polygon $\operatorname{PG}_k(\mathbf{p})(\cdot)$ as the continuous curve $\operatorname{PG}_k(\mathbf{p}) \colon \mathbb{R} \to \mathcal{M}$ such that
\[    \operatorname{PG}_k(\mathbf{p})(t) = M_{t2^{k}-n}(p_{n},p_{n+1}), \quad t\in [2^{-k}n,2^{-k}(n+1)), \quad n \in \mathbb{Z} , \quad k \in \mathbb{Z}_+  . \]
\end{definition}
We can now define the convergence of manifold-valued subdivision schemes intrinsically, in an analogous way to the definition in the case of real-valued subdivision schemes.
\begin{definition} \label{def:Convergence}
A manifold-valued subdivision scheme $\mathcal{M}$ is convergent if the sequence 
\[ \left\lbrace  \operatorname{PG}_k \left( \mathcal{T}^{k}(\mathbf{p}) \right) \right\rbrace_{k=0}^\infty \] 
converges uniformly in the metric of the manifold, \rm{for any sequence $\mathbf{p}$ of manifold data}.
\end{definition}
 We are now ready to prove the convergence result.
\begin{theorem} \label{thm:contraction_convergence}
Let $\mathcal{T}$ be a displacement-safe subdivision scheme for manifold data with a contractivity factor $\mu<1$. Then, $\mathcal{T}$ is convergent.
\end{theorem}
\begin{proof}
To show convergence we prove that $\{ \operatorname{PG}_k(\mathcal{T}^{k}(\mathbf{p}))(t) \}_{k\in \mathbb{N}}$ is a Cauchy sequence for all $t$ with a uniform constant. Since $\mathcal{M}$ is geodesically complete, it is also metric complete and any such Cauchy sequence converges to a limit in $\mathcal{M}$. 
 
Let $2^{-k}j \le t <2^{-k}(j+1)$, for some $j \in \mathbb{Z}$, then by the displacement-safe condition and the triangle inequality we get
\begin{eqnarray*}
d\left(\operatorname{PG}_k(\mathcal{T}^{k}(\mathbf{p})(t),\operatorname{PG}_{k+1}(\mathcal{T}^{k+1}(\mathbf{p}))(t)\right) 
&\le & d\left(\operatorname{PG}_k(\mathcal{T}^{k}(\mathbf{p})(t),\mathcal{T}^k(\mathbf{p})_j\right)  + d\left(\mathcal{T}^k(\mathbf{p})_j,\mathcal{T}^{k+1}(\mathbf{p})_{2j} \right) \\ & \: & +d\left(\mathcal{T}^{k+1}(\mathbf{p})_{2j},\operatorname{PG}_{k+1}(\mathcal{T}^{k+1}(\mathbf{p}))(t)\right) \\
&\le & \delta(\mathcal{T}^k(\mathbf{p})) + d\left(\mathcal{T}^{k+1}(\mathbf{p})_{2j},\mathcal{T}^k(\mathbf{p})_j \right)  +2\delta(\mathcal{T}^{k+1}(\mathbf{p})) \\
  &\le & (1+C+2 \mu)\delta(\mathcal{T}^k(\mathbf{p})) \le \widetilde{C} \mu^k ,
\end{eqnarray*}
where $\widetilde{C} = (1+C+2\mu)\delta(\mathbf{p})$ is a positive constant independent of $k$. The claims follows since $\mu<1$.
\end{proof}

In view of (b) of Remark \ref{rem:DS_condition}, we conclude
\begin{corollary} \label{cor:contractive_interpolation}
Assume that $\mathcal{T}$ is an interpolatory subdivision scheme of the form \eqref{eqn:interpolatory_adapted_ref}, defined on $(\mathcal{M},d)$, with a contractivity factor. Then, $\mathcal{T}$ is a convergent subdivision scheme.
\end{corollary}


\section{Adaptation based on geodesic inductive means} \label{sec:our_adaptation}
 
We study the adaptation of a given linear subdivision scheme to manifold-valued data. Our adaptation method is a specific choice in the second method in Subsection~\ref{subsec:Subdivision_schemes_based_RBA}. The expression of the refinement rules \eqref{eqn:TwoPartsLinearSchemes} in terms of repeated binary averages that we use is new and is designed in a way that facilitates the convergence analysis of the adapted schemes.

A basic property of all adaptations based on repeated geodesic averages, is that if one uses the arithmetic (binary) average for numbers instead of the geodesic average in the adapted refinement rules, the resulting refinement rules must coincide with those of the original linear subdivision schemes. We further demand the preservation of symmetry in the refinement rules, if any. Many families of subdivision schemes, e.g. \cite{dubucDeslauriers1}, consist of subdivision schemes with symmetric masks, namely with mask coefficients satisfying $a_i=a_{-i}$, $i\in\mathbb{Z}$. Our adaptation of the refinement rules takes into account this symmetry.


\subsection{The adaptation method} \label{subsec:ouradaptation}

Our adaptation of weighted averages is based on the idea of inductive means \cite{lim2014weighted}.
\begin{definition} \label{def:non_sym_avg}
Let $ \mathbf{p} =(p_1,\ldots,p_n)$ be a finite sequence of manifold elements, and let $ \mathbf{w} = ( w_1,\ldots,w_n)$ be their associated real weights satisfying $\sum_{j=1}^n w_j = 1$. We further assume that $ w_1 \ge w_2 \ge \ldots \ge w_n$. Then, the geodesic inductive mean $\rba{n}{\mathbf{p}}{\mathbf{w}}$ is defined recursively as,
\begin{equation} \label{eqn:def_of_inductive_mean}
\rba{n}{\mathbf{p}}{\mathbf{w}}  = 
\begin{cases} 
  M_{w_2}(p_1,p_2) 																											&\mbox{if } n = 2, \\  
  M_{w_n}( \rba{n-1}{ (p_1,\ldots,p_{n-1})}{ \frac{1}{1-w_n}( w_1,\ldots,w_{n-1} )},p_n) & \mbox{if } n >2 . 
\end{cases} 
\end{equation}
\end{definition}

It is easy to verify that Definition \ref{def:non_sym_avg}, when applied on real numbers with $M_t$ the binary arithmetic mean, is identical with averaging the entire set of numbers at once, since commutativity is valid. Therefore, the basic requirement of adaptation, as described above, is satisfied. 

It is interesting to note that in Hadamard spaces (NPC spaces), the inductive mean in \eqref{eqn:def_of_inductive_mean} approximates the Riemannian center of mass (mentioned in Subsection~\ref{subsec:Subdivision_schemes_based_RBA}), defined as
\begin{equation} \label{eqn:Karcher_Mean}
 \arg \min_{p \in \mathcal{M}} \sum_{j=1}^n w_j \left( d(p,p_j)\right)^2  .
\end{equation}
For manifolds or for metric spaces \eqref{eqn:Karcher_Mean} is not necessarily unique, and no explicit form of it is available. Yet, in Hadamard spaces \eqref{eqn:Karcher_Mean} is unique, and the rate of convergence of $\rba{n}{\mathbf{p}}{\mathbf{w}}$ to \eqref{eqn:Karcher_Mean} as $n \to \infty$ can be found in \cite{lim2013approximations}.

\begin{remark}
The weights of Definition \ref{def:non_sym_avg} are assumed to be sorted. The reason is to facilitate our calculations of contractivity. This statement is demonstrated through the examples in Section \ref{sec:examples} and their analysis. At this point, consider the recursive form of the inductive mean together with the triangle inequality to have  
\[  d \left( \rba{n}{\mathbf{p}}{\mathbf{w}}, q \right) \le   d\left(\rba{n}{\mathbf{p}}{\mathbf{w}},\rba{n-1}{\hat{\mathbf{p}} }{\hat{\mathbf{w}} } \right) + d\left( \rba{n-1}{\hat{\mathbf{p}} }{\hat{\mathbf{w}} } , q \right) , \quad q \in \mathcal{M} , \]
with $\hat{\mathbf{p}} = (p_1,\ldots,p_{n-1})$ and $\hat{\mathbf{w}}=\frac{1}{1-w_n}( w_1,\ldots,w_{n-1} )$. Then, in cases of positive weights we have by the metric property \eqref{eqn:metric_property} that the first distance is equal to $ w_n d\left(p_n,\rba{n-1}{\hat{\mathbf{p}} }{\hat{\mathbf{w}} } \right) $, regardless of $q$. To minimize this distance we require $w_n$ to be as small as possible.
\end{remark}

For the preservation of symmetry of the refinement rules we provide a symmetrical version of $\rba{n}{\mathbf{p}}{\mathbf{w}}$, denoted by $\srba{n}{\mathbf{p}}{\mathbf{w}}$ and defined as follows.
\begin{definition} \label{def:sym_avg}
Let $ \mathbf{p} =(p_1,\ldots,p_n)$ be a finite sequence of manifold elements, and let $ \mathbf{w} = ( w_1,\ldots,w_n)$ be their associated real weights satisfying $\sum_{j=1}^n w_j = 1$, 
and 
\[ w_j=w_{n-j+1}  , \quad j=1,\ldots, \ell  \quad , \quad \ell = \lfloor n/2 \rfloor. \]

For even $n$ we define the symmetric average as
\[  \srba{n}{\mathbf{p}}{\mathbf{w}} = M_{1/2}( \rba{n/2}{\mathbf{p}^1}{\mathbf{w}^1},\rba{n/2}{\mathbf{p}^2}{\mathbf{w}^1} , \]
where $\mathbf{w}^1$ is the sorted set of weights obtained from $(w_1, \ldots, w_{n/2})$ and $\mathbf{p}^1$ is their associated data points from $(p_1,\ldots, p_{n/2})$. Similarly, $\mathbf{p}^2$ is the data points from $(p_{n/2+1}, \ldots, p_{n})$ corresponding to $\mathbf{w}^1$.

If $n=2\ell+1$ then we redefine the weights to be of even length and symmetric by 
\[ \widetilde{\mathbf{w}} = (w_1,\ldots,w_{\ell-1},\frac{1}{2}w_{\ell},\frac{1}{2}w_{\ell},w_{\ell+1},\ldots,w_n) , \]
with the corresponding elements set as $\widetilde{\mathbf{p}} = (  p_1,\ldots,p_{\ell-1},p_{\ell},p_{\ell},p_{\ell+1},\ldots,p_n)$, and $ \srba{n}{\mathbf{p}}{\mathbf{w}}$ is defined as $ \srba{n+1}{\widetilde{\mathbf{p}}}{\widetilde{\mathbf{w}}}$.
\end{definition}

Equipped with Definitions \ref{def:non_sym_avg} and \ref{def:sym_avg}, we can formulate the adaptation method.
\begin{definition} \label{def:adapted:subdivision_schemes}
Let $\mathcal{S}$ be a linear univariate subdivision schemes, given by \eqref{eqn:TwoPartsLinearSchemes}, and let $M_t$ be a geodesic weighted average. For the adaptation $\mathcal{T}$ of the refinement rules of $\mathcal{S}$ we denote the local subset of the data $\mathbf{p}$ participating in the refinement rules for $\mathcal{T}(\mathbf{p})_{2j}$ and $\mathcal{T}(\mathbf{p})_{2j+1}$ by
\[  \mathbf{p}^j = \left(  p_{j-u},\ldots,p_{j+\ell}\right) \in \mathcal{M}^n   , \quad n= u+\ell+1  . \]
We denote the corresponding weights in the rule for $\mathcal{T}(\mathbf{p})_{2j}$ by $w_i = a_{2u+2-2i} $, $i=1,\ldots,n$ and in the rule for $\mathcal{T}(\mathbf{p})_{2j+1}$ by $\quad u_i = a_{2u+3-2i}$, $i = 1,\ldots,n$.
With these notations, the adapted refinement rules are 
\[ \mathcal{T}(\mathbf{p})_{2j} = \rba{n}{\widetilde{\mathbf{p}^j}}{\widetilde{\mathbf{w}}}  \qquad \text{  and  } \qquad  \mathcal{T}(\mathbf{p})_{2j+1} = \rba{n}{\widehat{\mathbf{p}^j}}{\widehat{\mathbf{u}}} , \]
where $\widetilde{\mathbf{w}}$ is the sorted $\mathbf{w} = \left(w_1,\ldots,w_n\right)$ and $\widetilde{\mathbf{p}^j}$ consists of the corresponding points to $\widetilde{\mathbf{w}}$ from $\mathbf{p}^j$. Similarly, $\widehat{\mathbf{u}}$ is the sorted $\mathbf{u}= \left(u_1,\ldots,u_n\right)$ and $\widehat{\mathbf{p}^j}$ consists of the corresponding points to $\mathbf{u}$ from $\mathbf{p}^j$. If in addition, there is a symmetry in the weights of the refinement rules, then the adapted refinement rule is defined by the symmetrical average of Definition \ref{def:sym_avg}. Namely, $\mathbf{w}_i = \mathbf{w}_{n-i+1} $, $i=1,\ldots,n$, leads to $\mathcal{T}(\mathbf{p})_{2j}  = \srba{n}{\mathbf{p}^j}{\mathbf{w}}$, and $\mathbf{u}_i = \mathbf{u}_{n-i+1} $, $i=1,\ldots,n$ implies $\mathcal{T}(\mathbf{p})_{2j+1}  = \srba{n}{\mathbf{p}^j}{\mathbf{u}}$.
\end{definition}

We term the schemes of Definition \ref{def:adapted:subdivision_schemes} (based on Geodesic Inductive Means) GIM-schemes.


\subsection{The GIM-schemes are displacement-safe}

For the the case of interpolatory schemes, we get by Corollary \ref{cor:contractive_interpolation} that contractivity implies convergence. However, for non-interpolatory schemes, contractivity by itself does not imply convergence but together with the displacement-safe condition \eqref{eqn:contractivity_extra_condition} in view of Theorem \ref{thm:contraction_convergence}. The following proposition reduces the proof of convergence of GIM-schemes to the proof of their contractivity.

\begin{proposition} \label{prop:distance_to_non_sym_RBA}
In the notation of Definition \ref{def:sym_avg}, we have
\[ \max_{i \in \{1,\ldots,n\}} d( \rba{n}{\mathbf{p}}{\mathbf{w}} , p_i ) \le C_n \delta(\mathbf{p}) , \]
where $C_n$ depends on $n$ and on $\|\mathbf{w}\|_{\infty} = \max_{1 \le j \le n} |w_{j}| $, but is independent of $\mathbf{p}$.
\end{proposition}
\begin{proof}
We prove the proposition by induction on $m$ in \eqref{eqn:def_of_inductive_mean}. In the $m$-th step, $2 \le m \le n $ in \eqref{eqn:def_of_inductive_mean} we use as weights the normalized partial set of the first $m$ weights,
\[  \mathbf{w}_m = \frac{1}{\sum_{j=1}^m w_j}(w_1,\ldots,w_m) ,\]
their associated set of elements $\mathbf{p}_m = (p_1,\ldots,p_m)$ and the corresponding 
\[ \delta \left(\mathbf{p}_m\right) = \max_{i=1,\ldots,m-1} d\left(p_i,p_{i+1} \right) .\] 
Clearly, $\delta(\mathbf{p}_m) \le \delta(\mathbf{p})$. 

The basis of the induction is $m=2$, where $w_1+w_2 = 1$. Then, by the metric property \eqref{eqn:metric_property}
\[  \max_{i=1,2}(d(M_{w_2}(p_1,p_2),p_i) \le \max \{|w_1|,|w_2|\} \delta(\mathbf{p}_2)  \le \| \mathbf{w} \|_\infty \delta(\mathbf{p}_2)  .\]
Thus, we can choose $C_2 = \| \mathbf{w} \|_{\infty}$. For the induction step, we assume
\begin{equation} \label{eqn:induction_assumption_local_ds}
 \max_{i \in \{1,\ldots,m\}} d( \rba{m}{\mathbf{p}_m}{\mathbf{w}_m}  , p_i ) \le C_m \delta(\mathbf{p_m}) , 
\end{equation}
for a fixed $m$, $2 \le m<n$. 

First, we bound the distance between the averages $\rba{m+1}{\mathbf{p}_{m+1}}{\mathbf{w}_{m+1}}$ and $\rba{m}{\mathbf{p}_m}{\mathbf{w}_m} $, which in view of Definition \ref{def:non_sym_avg} and \eqref{eqn:metric_property} is given by
\[  d(M_{\frac{w_{m+1}}{\sum_{j=1}^{m+1}w_j}}( \rba{m}{\mathbf{p}_m}{\mathbf{w}_m}, p_{m+1}  )   , \rba{m}{\mathbf{p}_m}{\mathbf{w}_m})  =  \abs{\frac{w_{m+1}}{\sum_{j=1}^{m+1}w_j}} d(\rba{m}{\mathbf{p}_m}{\mathbf{w}_m}, p_{m+1}  ) .\]
Now
\[ d(\rba{m}{\mathbf{p}_m}{\mathbf{w}_m}, p_{m+1}  ) \le  d(\rba{m}{\mathbf{p}_m}{\mathbf{w}_m}, p_{j}  ) + d(p_j,p_{m+1})  , \]
and since there exists $j$, $1 \le j \le m$ such that $d(p_j,p_{m+1})  \le \delta(\mathbf{p}_{m+1}) $, we get by the induction hypothesis \eqref{eqn:induction_assumption_local_ds}, and since $\delta(\mathbf{p}_m) \le \delta(\mathbf{p}_{m+1})$,
\begin{equation} \label{eqn:induction_auxiler}
d( \rba{m+1}{\mathbf{p}_{m+1}}{\mathbf{w}_{m+1}} , \rba{m}{\mathbf{p}_m}{\mathbf{w}_m}  ) \le  \abs{\frac{w_{m+1}}{\sum_{j=1}^{m+1}w_j}}(C_m+1)\delta(\mathbf{p}_{m+1}) .
\end{equation}
To bound $\abs{\frac{w_{m+1}}{\sum_{j=1}^{m+1}w_j}}$, recall that $\mathbf{w}$ is sorted. If $w_{m+1}<0$ then $\sum_{j=1}^{m+1} w_j \ge \sum_{j=1}^{n} w_j=1$, and therefore $\abs{\frac{w_{m+1}}{\sum_{j=1}^{m+1}w_j}} \le \abs{w_{m+1}} \le \norm{\mathbf{w}}_\infty$. On the other hand, if $w_{m+1} \ge 0$, then $ \abs{\frac{w_{m+1}}{\sum_{j=1}^{m+1}w_j}} \le  \frac{1}{m+1} $ since $\sum_{j=1}^{m+1} w_j \ge (m+1)w_{m+1}$. Thus,
\begin{equation} \label{eqn:bound_induction_factor}
\abs{\frac{w_{m+1}}{\sum_{j=1}^{m+1}w_j}} \le \max \left\lbrace \frac{1}{m+1} , \norm{\mathbf{w}}_\infty \right\rbrace  .
\end{equation}

Now, for any $1 \le j \le m+1$ we bound $d( \rba{m+1}{\mathbf{p}_{m+1}}{\mathbf{w}_{m+1}} , p_j)$ by
\[  d( \rba{m+1}{\mathbf{p}_{m+1}}{\mathbf{w}_{m+1}} , \rba{m}{\mathbf{p}_m}{\mathbf{w}_m}  ) + d(\rba{m}{\mathbf{p}_m}{\mathbf{w}_m},p_{j^\ast}) + d(p_{j^\ast},p_j) , \]
with $j^\ast$ satisfying $d(p_{j^\ast},p_j)  \le \delta(\mathbf{p}) $ and $1 \le j^\ast \le m$. Combining the latter with \eqref{eqn:induction_assumption_local_ds}, \eqref{eqn:induction_auxiler} and \eqref{eqn:bound_induction_factor} we obtain
\[d( \rba{m+1}{\mathbf{p}_{m+1}}{\mathbf{w}_{m+1}} , p_j)   \le  C_{m+1} \delta(\mathbf{p}) , \quad j \in \{ 1,\ldots,m+1 \} \]
where $C_{m+1} = \left( 1 + C_m \right) \left(1  + \max \left\lbrace \frac{1}{m+1} , \norm{ \mathbf{w}}|_\infty \right\rbrace \right)$, depends solely on $m$ and $\norm{\mathbf{w}}_\infty$ .
\end{proof}

We use the results of Proposition \ref{prop:distance_to_non_sym_RBA} to obtain a similar conclusion for $\srba{n}{\mathbf{p}}{\mathbf{w}}$.
\begin{corollary} \label{cor:sym_DS_condition}
In the notation of Definition \ref{def:non_sym_avg}, we have
\[ d( \srba{n}{\mathbf{p}}{\mathbf{w}} , \mathbf{p} ) = \max_{i \in \{1,\ldots,n\}} d( \srba{n}{\mathbf{p}}{\mathbf{w}} , p_i ) \le C^S_n \delta(\mathbf{p}) , \]
where $C^S_n = 2C_\ell+\frac{1}{2}$, $\ell = \lfloor \frac{n+1}{2} \rfloor$, and $C_\ell$ is the constant of Proposition \ref{prop:distance_to_non_sym_RBA}.
\end{corollary}
\begin{proof}
Using the notation of Definition \ref{def:sym_avg}, we denoted by $p_{j_1} \in \mathbf{p}^1$ and $p_{j_2} \in \mathbf{p}^2$ points that satisfy $d(p_{j_1},p_{j_2}) \le \delta(\mathbf{p})$ (such two points always exist). Without loss of generality, let $p_{i} \in \mathbf{p}^1$. Then, by the metric property \eqref{eqn:metric_property} and the triangle inequality we get
\begin{equation} \label{eqn:bound_sym_avg_proof}
\begin{array}{r@{}l} 
 d( \srba{n}{\mathbf{p}}{\mathbf{w}} , p_i )   &{} \le  d(\srba{n}{\mathbf{p}}{\mathbf{w}}, \rba{\ell}{\mathbf{p}^1}{\mathbf{w}^1}) + d(\rba{\ell}{\mathbf{p}^1}																													{\mathbf{w}^1} , p_i )  \\ 
 																	&{} \le  \frac{1}{2} d(\rba{\ell}{\mathbf{p}^1}{\mathbf{w}^1},\rba{\ell}{\mathbf{p}^2}{\mathbf{w}^1}) + d(\rba{\ell}{\mathbf{p}^1}																										{\mathbf{w}^1} , p_i )  .
\end{array}
\end{equation}
Now, due to Proposition \ref{prop:distance_to_non_sym_RBA}, we have
\[  d(\rba{\ell}{\mathbf{p}^1}{\mathbf{w}^1} , p_i ) \le C_\ell  \delta(\mathbf{p}^1) \le C_\ell  \delta(\mathbf{p}) , \]
while by the triangle inequality and by Proposition \ref{prop:distance_to_non_sym_RBA} we get
\begin{eqnarray*}
 d(\rba{\ell}{\mathbf{p}^1}{\mathbf{w}^1},\rba{\ell}{\mathbf{p}^2}{\mathbf{w}^1}) &\le & d(\rba{\ell}{\mathbf{p}^1}{\mathbf{w}^1},p_{j_1}) + d(p_{j_1},p_{j_2}) + d(p_{j_2},\rba{\ell}{\mathbf{p}^2}{\mathbf{w}^1})  \\
 &\le & C_\ell  \delta(\mathbf{p}) +  \delta(\mathbf{p}) +C_\ell  \delta(\mathbf{p})  = (2C_\ell +1) \delta(\mathbf{p}) .
\end{eqnarray*}
The last two bounds together with \eqref{eqn:bound_sym_avg_proof} complete the proof.
\end{proof}

Proposition \ref{prop:distance_to_non_sym_RBA} and Corollary \ref{cor:sym_DS_condition} lead to
\begin{corollary} \label{cor:local_DS_for_GIM}
Any GIM-subdivision scheme satisfies the displacement-safe condition \eqref{eqn:contractivity_extra_condition}.
\end{corollary}

Therefore, in view of Corollary \ref{cor:local_DS_for_GIM} and Theorem \ref{thm:contraction_convergence} we conclude.
\begin{theorem} \label{thm:local_convergence_by_contractivity}
Let $\mathcal{T}$ be a GIM-subdivision scheme. If $\mathcal{T}$ has a contractivity factor then $\mathcal{T}$ is convergent.
\end{theorem}

\begin{remark}
Due to contractivity, the convergence from all initial data is also valid for spaces where the geodesic curve is not unique, regardless of the choice of $M_t$. In other words, the freedom in choosing the geodesic on which we define the geodesic average, is reflected by a set of possible limits (a number of possible limits for each initial data) but not in the fact that the limit exists. Note that since $\mathcal{M}$ is a geodesic complete manifold, the injectivity radius of the manifold is bounded away from zero, meaning that from some fixed refinement level, the geodesic is guaranteed to be unique.
\end{remark}

\section{Examples of convergent GIM-schemes} \label{sec:examples}

The aim of this subsection is two folded; First, to demonstrate via examples our adaptation method. Second, to present a technique for deriving a contractivity factor of a GIM-scheme. 

Let us begin with the adaptation of the family of interpolatory $4$-point schemes \cite{4Point}.
\begin{example} \label{exm:maifold_4pt}   
The interpolatory $4-$point scheme \cite{4Point} is defined in the functional setting as
\begin{equation} \label{def:numbers4pt}
  \mathcal(S(\mathbf{f}))_{2i} = f_i  , \qquad\quad\text{and}\quad\qquad
    \mathcal(S(\mathbf{f}))_{2i+1} = -\omega(f_{i-1}+f_{i+2}) + (\frac{1}{2}+\omega)(f_i+f_{i + 1}) .
\end{equation} 
With $\omega \in (0,\omega^\ast)$ and $\omega^\ast \left(\approx  0.19273\right)$ the unique solution of the cubic equation $32\omega^3+4\omega-1= 0$, the limits generated by the scheme are $C^1$ \cite{hechler2009c1}. The case $\omega=\frac{1}{16}$ coincides with the cubic Dubuc-Deslauriers scheme \cite{dubucDeslauriers1}. 

We adapt the $4$-point scheme using a geodesic average $M_t$, under the assumption that it is well defined for $t$ in a small neighbourhood of $[0,1]$. Note that such an adaptation was already done in \cite{UriNir} for positive definite matrices, and in \cite{kels2013subdivision} for sets. The symmetry of the coefficients, $(-\omega,\frac{1}{2}+\omega,\frac{1}{2}+\omega,-\omega)$ implies that the adaptation of \eqref{def:numbers4pt} is
\begin{equation} \label{eqn:adapted4pt_local}
 \mathcal{T}(\mathbf{p})_{2i} = p_i  , \quad\text{and}\quad \mathcal{T}(\mathbf{p})_{2i+1} =  \srba{4}{\mathbf{p}}{\mathbf{w}}  = M_\frac{1}{2} \left( M_{-2\omega} \left( p_{i} ,p_{i-1}  \right), M_{-2\omega} \left(p_{i+1}, p_{i+2} \right) \right) , 
\end{equation}
with $\mathbf{w} = (-\omega,\frac{1}{2}+\omega,\frac{1}{2}+\omega,-\omega)$ and $\mathbf{p} = (p_{i-1},p_i,p_{i+1},p_{i+2})$.
\begin{figure}
\centering    
\subfloat[Data points]{\label{fig:4pt_points}			   \includegraphics[width=0.4\textwidth]{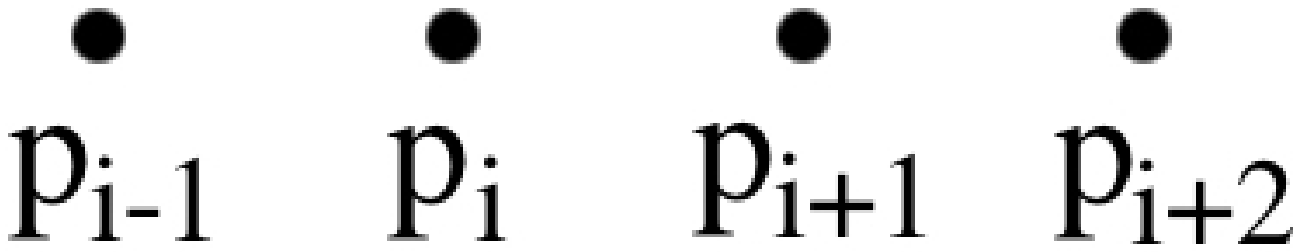}} \quad
\subfloat[$M_{-2\omega} \left( p_{i} ,p_{i-1}  \right)$ and $M_{-2\omega} \left(p_{i+1}, p_{i+2} \right)$]{\label{fig:4pt_first_level_V2} \includegraphics[width=0.45\textwidth]{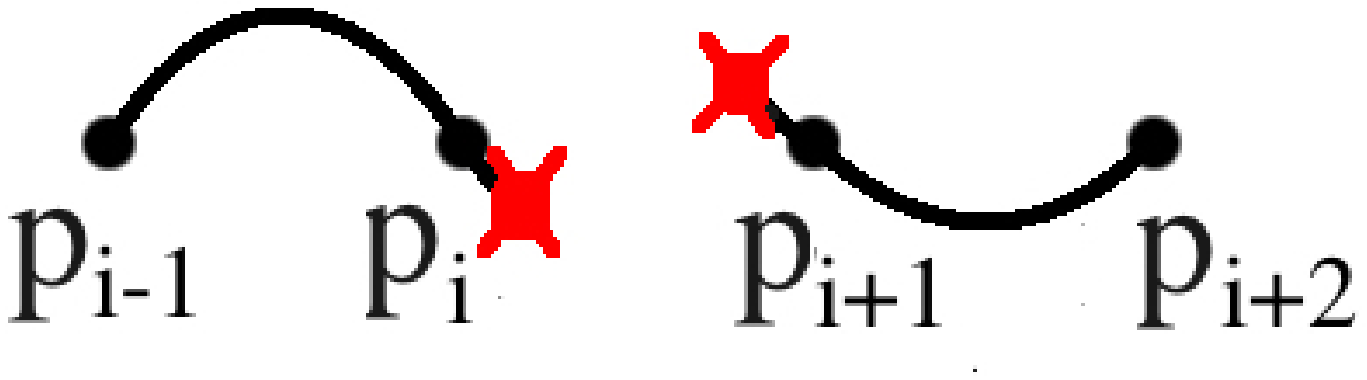}} \\
\subfloat[Inserted point]{\label{fig:4pt_new_point}      \includegraphics[width=0.45\textwidth]{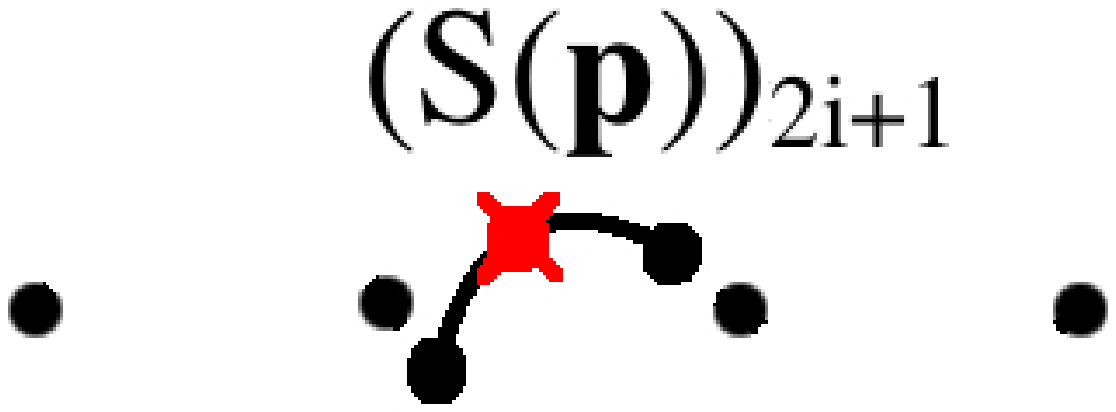}}
\subfloat[Distance to bound: the bright curve connects the two refined points]{\label{fig:4pt_dist2bound}    \includegraphics[width=0.45\textwidth]{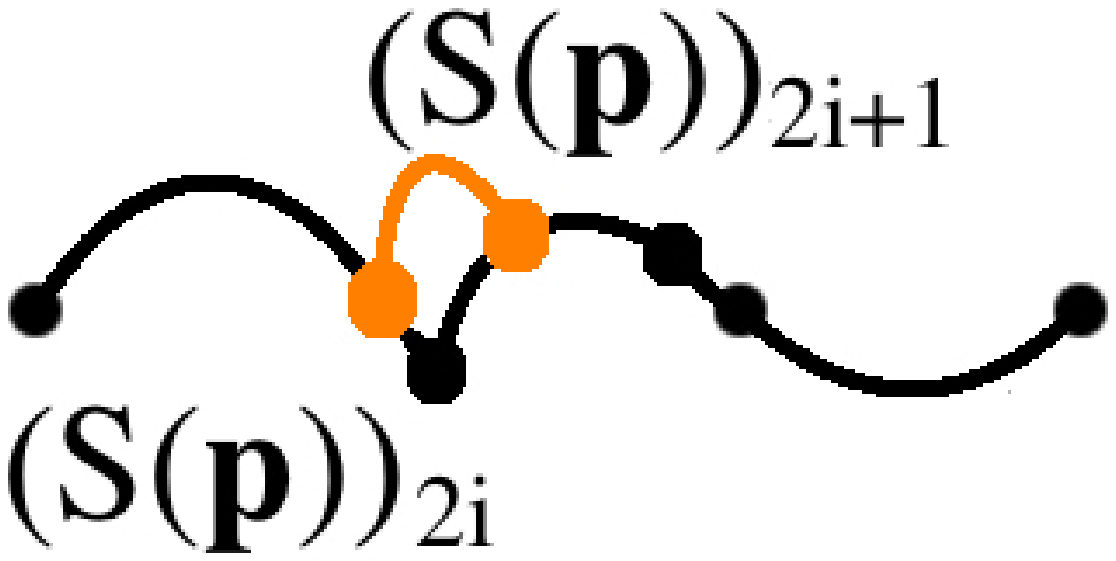}} 
\caption{The adaptation of the $4$-point scheme. The curved lines represent geodesic curves connecting two points}
\label{fig:4pt}
\end{figure}
 
The refinement \eqref{eqn:adapted4pt_local} is presented schematically in Figures~\ref{fig:4pt_points}--\ref{fig:4pt_new_point}. The analysis of contractivity aims to bound the distance $d( \mathcal{T}(\mathbf{p})_{2i+1}, \mathcal{T}(\mathbf{p})_{2i})$, which is depicted in Figure~\ref{fig:4pt_dist2bound}. 

By the triangle inequality and the metric property of $M_t$ \eqref{eqn:metric_property} we have (see Figure \ref{fig:4pt_dist2bound})
\begin{equation} \label{eqn:4pt_bound_ineq}
 d( \mathcal{T}(\mathbf{p})_{2i+1}, \mathcal{T}(\mathbf{p})_{2i}) \le d(\mathcal{T}(\mathbf{p})_{2i+1},M_{-2\omega} \bigl( p_{i} ,p_{i-1}  \bigr) ) + d(M_{-2\omega} \bigl( p_{i} ,p_{i-1}  \bigr),p_{i}) ,
\end{equation}
with,
\begin{equation}  \label{eqn:4pt_first_dist}
d(M_{-2\omega} \bigl( p_{i} ,p_{i-1}  \bigr),p_{i}) = 2\omega d(p_{i-1},p_i) \le  2\omega \delta(\mathbf{p}) 
\end{equation}
and  
\begin{equation} \label{eqn:4pt_sec_dist}
d(\mathcal{T}(\mathbf{p})_{2i+1},M_{-2\omega} \bigl( p_{i} ,p_{i-1}  \bigr)) = \frac{1}{2} d(M_{-2\omega} \bigl(p_{i+1}, p_{i+2} \bigr) ,M_{-2\omega} \bigl( p_{i} ,p_{i-1}  \bigr)) .
\end{equation}
Using again the triangle inequality, we bound the right-hand side of \eqref{eqn:4pt_sec_dist},
\begin{eqnarray*}
d( M_{-2\omega} \bigl( p_{i} ,p_{i-1}  \bigr), M_{-2\omega} \bigl(p_{i+1}, p_{i+2} \bigr) ) &\le &  d(M_{-2\omega} \bigl( p_{i} ,p_{i-1}  \bigr),p_i) \\ 
 &+& d(p_i,p_{i+1}) + d(p_{i+1},M_{-2\omega} \bigl(p_{i+1}, p_{i+2} \bigr) )  ,
\end{eqnarray*}
which in view of \eqref{eqn:4pt_first_dist} and \eqref{eqn:4pt_sec_dist} leads to
\begin{equation} \label{eqn:4pt_third_dist}
 d(\mathcal{T}(\mathbf{p})_{2i+1},M_{-2\omega} \bigl( p_{i} ,p_{i-1}  \bigr)) \le \frac{1}{2} (1+4 \omega)\delta(\mathbf{p}) . 
\end{equation}
Finally, using \eqref{eqn:4pt_bound_ineq},\eqref{eqn:4pt_first_dist} and \eqref{eqn:4pt_third_dist} we arrive at
\[ d(\mathcal{T}(\mathbf{p})_{2i+1},\mathcal{T}(\mathbf{p})_{2i}) \le (4\omega+\frac{1}{2}) \delta(\mathbf{p}) . \]
Due to the symmetry of the refinement rule \eqref{eqn:adapted4pt_local} we also have $d( \mathcal{T}(\mathbf{p})_{2i+1}, \mathcal{T}(\mathbf{p})_{2i+2})\le \mu \delta(\mathbf{p})$, corresponding to $\mu =4\omega+\frac{1}{2}$. Thus, when $ \mu <1 $ we have contractivity. Applying Corollary \ref{cor:contractive_interpolation} for $\omega<\frac{1}{8}$ we get convergence. Note that for the important case $\omega = \frac{1}{16}$ we have $\mu = \frac{3}{4}$, and that the best contractivity factor $\mu = \frac{1}{2}$ is obtained for the piecewise geodesic scheme with $\omega = 0$.
\end{example}

The next example derives the GIM-scheme from the interpolatory $6$-point Dubuc-Deslauriers (DD) scheme \cite{dubucDeslauriers1}.
\begin{example} \label{exm:manifold_6pt}
The interpolatory $6-$ point DD scheme is defined in the functional setting as
\begin{equation} \label{def:numbers6pt}
  \mathcal(S(\mathbf{f}))_{2i} = f_i  , \qquad\quad\text{and}\quad\qquad
    \mathcal(S(\mathbf{f}))_{2i+1} = \sum_{j=-2}^3  w_{j} f_{i+j} ,
\end{equation} 
where 
\[   (w_{-2},\ldots,w_3) = \frac{1}{2^8}\left( 3,-25,150,150,-25,3 \right) .\]
Thus, the adapted scheme is
\begin{equation} \label{def:manifold_6pt}
  \mathcal{T}(\mathbf{p})_{2i} = p_i  , \qquad\quad\text{and}\quad\qquad
    \mathcal{T}(\mathbf{p})_{2i+1} = \srba{6}{(p_{i-2},p_{i-1},\ldots,p_{i+3})}{\mathbf{w}} ,
\end{equation} 
where $\mathbf{w} = \frac{1}{2^8}\left( 3,-25,150,150,-25,3 \right)$. The analysis of contractivity of the adapted $6$-point \eqref{def:manifold_6pt} is given in Appendix \ref{appendix:6pt}. This analysis shows a contractivity factor of $\mu = 0.9844$. The convergence of this scheme follows by Corollary \ref{cor:contractive_interpolation}.
\end{example}

The last two examples demonstrate that as the support of the weights of the adapted refinement rule becomes large the derivation of a contractivity factor with the above tools becomes more difficult. Indeed, in a similar fashion and without any further assumptions on the metric space we do not get contractivity for the $8$-point DD subdivision scheme, adapted according to Definition \ref{def:adapted:subdivision_schemes}. It should be noted that the $8$-point DD scheme adapted by the log-exp mapping has a contractivity factor in Complete Riemannian manifolds \cite{wallnerconvergent}

We conclude this section with applications of Definition \ref{def:adapted:subdivision_schemes} to the adaptation of the first four B-spline schemes.
\begin{example} \label{exm:B_splines}
The mask of the B-spline subdivision scheme of degree $m$ has the nonzero coefficients 
\[ 2^{-m} \left( \binom{m+1}{0}, \binom{m+1}{1} , \ldots ,\binom{m+1}{m},\binom{m+1}{m+1} \right) .\] 
The GIM-scheme corresponding to $m=1$ generates the piecewise geodesic curve connecting consecutive initial points by geodesic curves. The adaptation of the next scheme, corresponding to $m=2$ (the corner cutting scheme) yields
\begin{equation} \label{def:CC_scheme}
  \mathcal{T}(\mathbf{p})_{2i} = M_\frac{1}{4}(p_i,p_{i+1})   , \qquad\quad\text{and}\quad\qquad
    \mathcal{T}(\mathbf{p})_{2i+1} =  M_\frac{3}{4}(p_i,p_{i+1}) ,
\end{equation} 
The refined points are inserted on the geodesic curve connecting consecutive points in $\mathbf{p}$ and it is easy to verify a contractivity factor $\frac{1}{2}$. The smoothness of this scheme is studied extensively in \cite{noakes1998nonlinear,noakes1999accelerations}. This scheme, for the manifold of positive definite matrices, is studied in \cite{NirUri}, and various algebraic properties of the limits generated by it are derived.

We adapt the cubic B-spline scheme ($m=3$), using the symmetrical mean of Definition \ref{def:sym_avg}. Thus
\begin{equation} \label{def:Cubic_B_spline_scheme}
  \mathcal{T}(\mathbf{p})_{2i} = \srba{3}{ (p_{i-1},p_i,p_{i+1})}{(\frac{1}{8},\frac{3}{4},\frac{1}{8})}  ,  \qquad
    \mathcal{T}(\mathbf{p})_{2i+1} =    M_\frac{1}{2}(p_i,p_{i+1}) .
\end{equation} 

\begin{figure}
\centering    
\subfloat[Data points]{\label{fig:CBS_points}						 		    \includegraphics[width=0.4\textwidth]{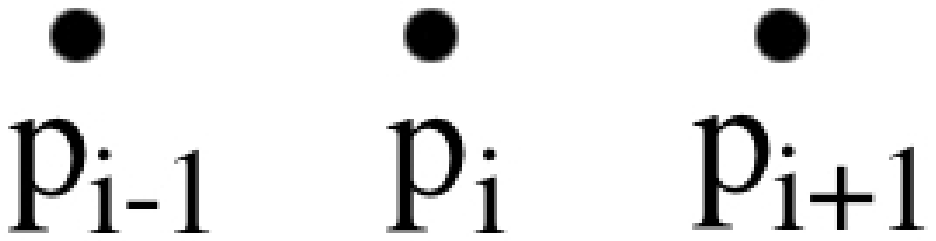}}
\subfloat[Inserted point (odd index)]{\label{fig:CBS_odd_refined} 		   		    \includegraphics[width=0.4\textwidth]{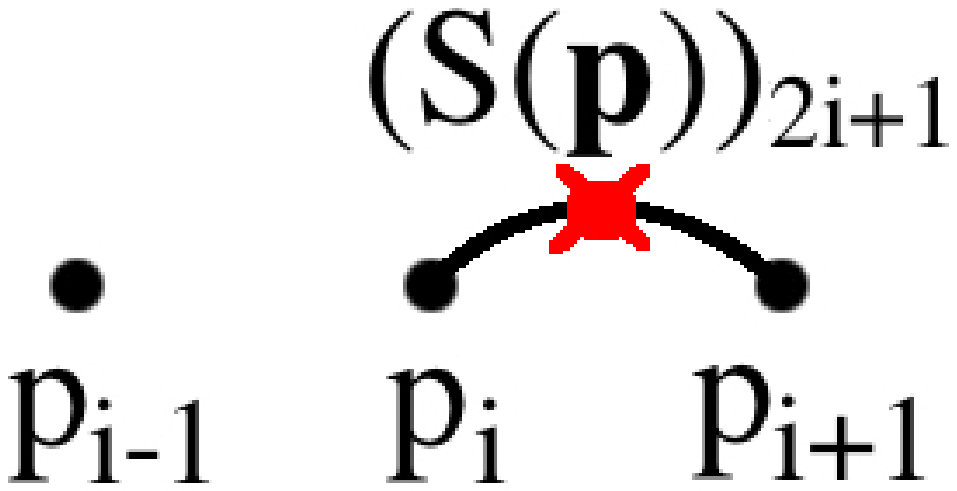}} \\
\subfloat[Refined point (even index)]{\label{fig:CBS_even_refined}        		    \includegraphics[width=0.4\textwidth]{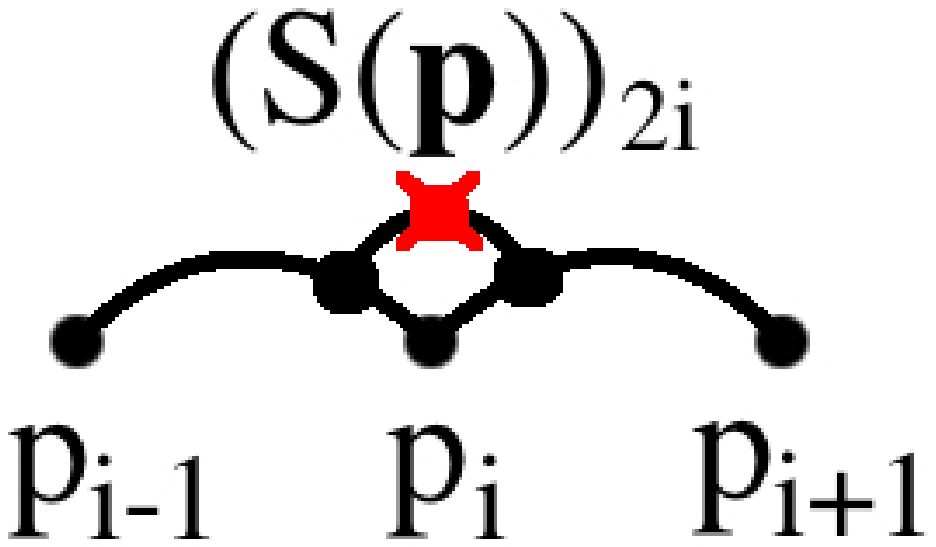}}
\subfloat[Distance to be bounded (length of the bright curve)]{\label{fig:CBS_dist2bound}    \includegraphics[width=0.4\textwidth]{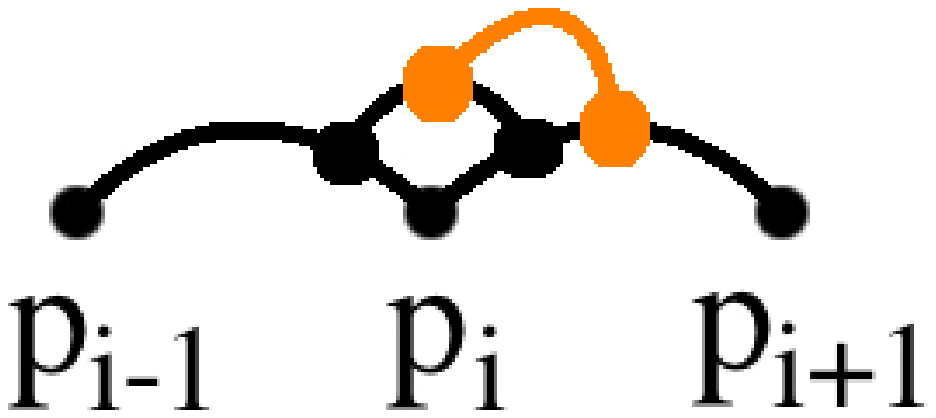}} 
\caption{The adaptation of the cubic B-spline scheme. The geodesic curves connecting two points are represented by curved lines}
\label{fig:CBS}
\end{figure}

Figure~\ref{fig:CBS} shows the refinement rules \eqref{def:Cubic_B_spline_scheme}, and the distance which we aim to bound in order to guarantee a contractivity factor. The explicit form of $\mathcal{T}(\mathbf{p})_{2i}$ is $M_\frac{1}{2} \left(  M_\frac{3}{4}(p_{i-1},p_i),M_\frac{1}{4}(p_{i},p_{i+1}) \right)$. $\mathcal{T}(\mathbf{p})_{2i}$ is depicted schematically in Figure~\ref{fig:CBS_even_refined} and $\mathcal{T}(\mathbf{p})_{2i+1}$ in Figure~\ref{fig:CBS_odd_refined}. To bound the refined distance $d(\mathcal{T}(\mathbf{p})_{2i},\mathcal{T}(\mathbf{p})_{2i+1})$ we first obtain
\[ d(\mathcal{T}(\mathbf{p})_{2i+1},M_\frac{1}{4}(p_{i},p_{i+1})) = d(M_\frac{1}{2}(p_i,p_{i+1}),M_\frac{1}{4}(p_{i},p_{i+1})) = \frac{1}{4}d(p_{i},p_{i+1}), \]
as both points in the intermediate distance are on the same geodesic. Thus, we have by the metric property of $M_t$ (see Figure~\ref{fig:CBS_dist2bound}),
\begin{eqnarray*}
d(\mathcal{T}(\mathbf{p})_{2i},\mathcal{T}(\mathbf{p})_{2i+1}) &\le & d(\mathcal{T}(\mathbf{p})_{2i}, M_\frac{1}{4}(p_{i},p_{i+1}) ) + d(M_\frac{1}{4}(p_{i},p_{i+1}), \mathcal{T}(\mathbf{p})_{2i+1} ) \\
													&\le & \frac{1}{2} d(M_\frac{3}{4}(p_{i-1},p_i),M_\frac{1}{4}(p_{i},p_{i+1})) + \frac{1}{4}\delta(\mathbf{p})  \\
													& \le & \frac{1}{2} \left( d(M_\frac{3}{4}(p_{i-1},p_i),p_i) + d(p_i,M_\frac{1}{4}(p_{i},p_{i+1}))  \right) + \frac{1}{4}\delta(\mathbf{p})  \\
													& \le & \left[ \frac{1}{2} \left( \frac{1}{4} + \frac{1}{4}\right) + \frac{1}{4} \right] \delta(\mathbf{p}) = \frac{1}{2}\delta(\mathbf{p}) .
\end{eqnarray*}
By symmetry we have the same result for $d(\mathcal{T}(\mathbf{p})_{2i-1},\mathcal{T}(\mathbf{p})_{2i}) $ and therefore a contractivity factor $\mu = \frac{1}{2}$ is obtained.

The last B-spline scheme considered in this example is the quartic B-spline ($m=4$). The adapted scheme is
\begin{equation} \label{eqn:quartic_B_spline_1}
\mathcal{T}(\mathbf{p})_{2i} =  \rba{3}{(p_i,p_{i+1},p_{i-1})}{( \frac{10}{16},\frac{5}{16},\frac{1}{16})} = M_\frac{15}{16} \left( p_{i-1}, M_\frac{1}{3}(p_{i},p_{i+1})   \right) ,
\end{equation}
and
\begin{equation} \label{eqn:quartic_B_spline_2}
\mathcal{T}(\mathbf{p})_{2i+1} = \rba{3}{  (p_{i+1},p_i,p_{i+2})}{(\frac{10}{16},\frac{5}{16},\frac{1}{16})}= M_\frac{15}{16} \left( p_{i+2},  M_\frac{2}{3}(p_{i},p_{i+1})  \right) .
\end{equation}
The contractivity analysis is presented in Appendix \ref{appendix:quartic_Bspline}, where a contractivity factor $\mu =\frac{5}{6}$ is established.

Recall that by Theorem \ref{thm:local_convergence_by_contractivity}, for all the B-spline schemes presented in this example, the contractivity implies convergence.
\end{example}

Example \ref{exm:B_splines} presents the analysis of the adaptation of the first four B-spline schemes. This analysis results in the convergence of the adapted schemes. Nevertheless, similar to the interpolatory case, the above analysis fails to obtain contractivity for schemes with a mask of large support. Indeed, for the quintic B-spline ($m=5$) we did not achieve a contractivity factor. Since most applicable, popular linear schemes have masks of relatively small supported, we are encouraged to construct their GIM-schemes and to analyze their convergence by the tools presented in this section.


\bibliography{manifold_GIM_bib}
\bibliographystyle{plain} 


\appendix  \label{appendix}

\section{Proofs of contractivity}  


\subsection{The contracivity of the adapted \texorpdfstring{$6$}--point scheme} \label{appendix:6pt}

Similar to the analysis of the $4$-point scheme of Example \ref{exm:maifold_4pt}, we aim to bound 
 \[ d(\mathcal{T}(\mathbf{p})_{2i+1},\mathcal{T}(\mathbf{p})_{2i}) .\]  
By Definitions \ref{def:non_sym_avg} and \ref{def:sym_avg} we have that $\srba{6}{(p_{i-2},\ldots,p_{i+3}  )}{\mathbf{w}} = M_\frac{1}{2}(M^1,M^2) $ where
\begin{equation*}  
M^1 = \rba{3}{(p_{i},p_{i-2},p_{i-1})}{\frac{1}{2^8}(150,3,-25)} =  M_{-25/128}( M_{3/153}(p_{i},p_{i-2}),p_{i-1} ) , 
\end{equation*}
and
\[ M^2 = \rba{3}{(p_{i+1},p_{i+3},p_{i+2})}{\frac{1}{2^8}(150,3,-25)} =  M_{-25/128}( M_{3/153}(p_{i+1},p_{i+3}),p_{i+2} )  .\]
Then, by the triangle inequality and the metric property we get from \eqref{def:manifold_6pt}
\begin{eqnarray*}
 d(\mathcal{T}(\mathbf{p})_{2i+1},\mathcal{T}(\mathbf{p})_{2i}) &\le & d(\mathcal{T}(\mathbf{p})_{2i+1},M^1) + d(M^1,p_{i}) \\
 & \le & \frac{1}{2}d(M^1,M^2) +  d(M^1,p_{i}) \\
 & \le & \frac{1}{2} \left( d(M^1,p_i) + d(p_i,p_{i+1}) + d(p_{i+1},M^2) \right) +  d(M^1,p_{i}) 
\end{eqnarray*}
which leads to
\begin{equation} \label{eqn:6pt_analysis_bound}
d(\mathcal{T}(\mathbf{p})_{2i+1},\mathcal{T}(\mathbf{p})_{2i}) \le \frac{1}{2}\delta(\mathbf{p}) + \frac{3}{2} d(M^1,p_i)  + \frac{1}{2} d(p_{i+1},M^2) .
\end{equation}
Now, $d(M^1,p_i) \le d(M^1,M_{3/153}(p_{i},p_{i-2} ) )+ d(M_{3/153}(p_{i},p_{i-2}),p_i )$, and by the metric property
\begin{equation} \label{eqn:6pt_first_avg_bnd}
d(M_{3/153}(p_{i},p_{i-2}),p_i ) \le \frac{3}{153}2\delta(\mathbf{p}) . 
\end{equation}
For the other distance we have
\[d(M^1,M_{3/153}(p_{i},p_{i-2}) ) = \frac{25}{128} d(p_{i-1},M_{3/153}(p_{i},p_{i-2}) )  , \]
and in view of \eqref{eqn:6pt_first_avg_bnd} 
\[  d(p_{i-1},M_{3/153}(p_{i},p_{i-2}) ) \le d(p_{i-1},p_{i}  ) + d(p_{i},M_{3/153}(p_{i},p_{i-2}) ) \le \delta(\mathbf{p}) + \frac{2}{51} \delta(\mathbf{p}). \]
Thus, 
\[   d(M^1,p_i) \le  (\frac{25}{128} \cdot \frac{53}{51}  + \frac{2}{51} )\delta(\mathbf{p})  < 0.2422 \delta(\mathbf{p}) . \]
By symmetry we also have $d(M^2,p_{i+1}) < 0.2422 \delta(\mathbf{p})$. Hence, 
\[ \frac{3}{2} d(M^1,p_i)  + \frac{1}{2} d(p_{i+1},M^2) < 2 \cdot 0.2422 \delta(\mathbf{p}) = 0.4844 \delta(\mathbf{p}) , \]
The contractivity factor $0.9844$ is revealed by using \eqref{eqn:6pt_analysis_bound}.


\subsection{The contracivity of the adapted quartic B-spline scheme} \label{appendix:quartic_Bspline}

In the case of the quartic B-Spline we have to bound the two distances $d\left( \mathcal{T}(\mathbf{p})_{2i}, \mathcal{T}(\mathbf{p})_{2i+1} \right)$ and $d\left(\mathcal{T}(\mathbf{p})_{2i-1} , \mathcal{T}(\mathbf{p})_{2i} \right)$ separately since no symmetry can be used here. The two bounds are obtained from \eqref{eqn:quartic_B_spline_1} and \eqref{eqn:quartic_B_spline_2} by the triangle inequality and the metric property \eqref{eqn:metric_property},
\begin{eqnarray*}
d( \mathcal{T}(\mathbf{p})_{2i}, \mathcal{T}(\mathbf{p})_{2i+1} ) &\le & d( \mathcal{T}(\mathbf{p})_{2i}, M_\frac{1}{3}(p_{i},p_{i+1}) ) + d(M_\frac{1}{3}(p_{i},p_{i+1}),M_\frac{2}{3}(p_{i},p_{i+1}) ) \\
&+&  d( M_\frac{2}{3}(p_{i},p_{i+1})  ,  \mathcal{T}(\mathbf{p})_{2i+1} ) \\
&\le & \frac{1}{16} d(p_{i-1}, M_\frac{1}{3}(p_{i},p_{i+1}) ) + \frac{1}{3} \delta(\mathbf{p}) + \frac{1}{16}  d( M_\frac{2}{3}(p_{i},p_{i+1})  ,  p_{i+2} ) \\
													&\le & \frac{1}{16} \cdot \frac{4}{3} \delta(\mathbf{p}) + \frac{1}{3} \delta(\mathbf{p}) + \frac{1}{16}\cdot \frac{4}{3} \delta(\mathbf{p}) = \frac{1}{2} \delta(\mathbf{p}) .
\end{eqnarray*}
And
\begin{eqnarray*}
d(\mathcal{T}(\mathbf{p})_{2i-1} , \mathcal{T}(\mathbf{p})_{2i} ) &\le & d(\mathcal{T}(\mathbf{p})_{2i-1}, M_\frac{2}{3}(p_{i-1},p_{i} ) + d(M_\frac{2}{3}(p_{i-1},p_{i}) ,p_i)  \\
&+ &    d(p_i,M_\frac{2}{3}(p_{i},p_{i+1})) + d(M_\frac{2}{3}(p_{i},p_{i+1}),\mathcal{T}(\mathbf{p})_{2i}) \\
&\le & 2( \frac{1}{16} \cdot \frac{4}{3} \delta(\mathbf{p})  + \frac{1}{3} \delta(\mathbf{p}) ) =  \frac{5}{6} \delta(\mathbf{p}) .
\end{eqnarray*}
Thus, the contractivity factor is $\mu =\frac{5}{6}$.


\end{document}